\documentclass[12pt]{amsart}
\usepackage{amsfonts}
\numberwithin{equation}{section}
\usepackage{amssymb}
\usepackage{amsmath}

\def\arg{\operatorname{arg}}
\def\meas{\operatorname{meas}}
\newtheorem{lem}{Lemma}[section]
\newtheorem{thm}{Theorem}[section]
\newtheorem{que}{Question}[section]

\theoremstyle{definition}
\newtheorem{defi}{Definition}[section]
\theoremstyle{remark}

\title[T direction]{Hayman $T$ directions of meromorphic functions in some angular domains}
\thanks{The work is supported by NSF of China (No.10871108)}
\author[Wu]{Wu Nan$^{1}$}
\author[Xuan]{Xuan Zu-Xing*$^{1,2}$}
\thanks{*corresponding author}
\address{$^{1}$Department of Mathematical Sciences, Tsinghua University, Beijing,
100084, People's Republic of China}
\address{$^{2}$Basic Department, Beijing Union University, No.97 Bei Si Huan Dong
Road, Chaoyang District, Beijing, 100101, People's Republic of
China} \email{wunan07@gmail.com} \email{xuanzuxing@ss.buaa.edu.cn}
\date{\today, Preliminary version}
\begin{document}

\begin{abstract}
This paper is devoted to investigate the singular directions of
meromorphic functions in some angular domains. We will confirm the
existence of Hayman $T$ directions in some angular domains. This is
a continuous work of Yang [Yang L., Borel directions of meromorphic
functions in an angular domain, Science in China, Math.
Series(I)(1979), 149-163.] and Zheng [Zheng, J.H., Value
Distribution of Meromorphic Functions, preprint.].
\end{abstract}
\maketitle 

\maketitle

\section{Introduction and Main Results}
\setcounter{equation}{0} Let $f(z)$ be a meromorphic function on the
whole complex plane. We will use the standard notation of the
Nevanlinna theory of meromorphic functions, such as $T(r,f), N(r,f),
m(r,f), \delta(a,f)$. For the detail, see \cite{Yang}. The order and
lower order of it are defined as follows
$$\lambda(f)=\limsup\limits_{r\rightarrow\infty}\frac{\log T(r,f)}{\log r}$$
and
$$\mu(f)=\liminf\limits_{r\rightarrow\infty}\frac{\log T(r,f)}{\log r}.$$

In view of the second fundamental theorem of Nevanlinna, Zheng
\cite{Zheng} introduced a new singular direction, which is named $T$
direction.\\
\begin{defi}\label{def2} A direction $L: \arg z=\theta$ is called a
$T$ direction of $f(z)$ if for any $\varepsilon>0$, we have
\begin{equation*}
\limsup\limits_{r\rightarrow\infty}\frac{N(r,Z_\varepsilon(\theta),f=a)}{T(r,f)}>0
\end{equation*}
for all but at most two values of $a$ in the extended complex plane
$\widehat{\mathbb{C}}$. Here
$$N(r,\Omega,f=a)=\int_1^r{n(t,\Omega,f=a)\over t}dt,$$
where $n(t,\Omega,f=a)$ is the number of the roots of $f(z)=a$ in
$\Omega\cap\{1<|z|<t\},$ counted according to multiplicity. And
through out this paper, we denote
$Z_\varepsilon(\theta)=\{z:\theta-\varepsilon<\arg
z<\theta+\varepsilon\}$ and $\Omega(\alpha,\beta)=\{z:\alpha<\arg
z<\beta\}$.
\end{defi}

The reason about the name is that we use the Nevanlinna's
characteristic $T(r,f)$ as comparison body. Under the growth
condition
\begin{equation}\limsup\limits_{r\rightarrow\infty}\frac{T(r,f)}{(\log
r)^2}=+\infty. \end{equation} Guo, Zheng and Ng \cite{GuoZhengNg}
confirmed the existence of this type direction and they pointed out
the growth condition (1.1) is sharp. Later, Zhang \cite{zhang}
showed that $T$ directions are different from Borel directions whose
definition can be found in \cite{Hayman}.

In 1979, Yang \cite{Yang01} showed the following theorem, which says that the condition for an angular domain to contain at least one Borel direction.\\
\textbf{Theorem A.} \emph{Let $f(z)$ be a meromorphic function on
the whole complex plane, with $\mu<\infty,0<\lambda\leq\infty$. Let
$\rho$ be a finite number such that $\lambda\geq\rho\geq\mu$ and
$\rho>1/2$. If $f^{(k)}(z)(k\geq0)$ has $p$ distinct deficient
values $a_1,a_2,\cdots,a_p$, then in any angular domain
$\Omega(\alpha,\beta)$ such that
$$\beta-\alpha>\max\{\frac{\pi}{\rho},2\pi-\frac{4}{\rho}\sum\limits_{i=1}^p\arcsin\sqrt{\frac{\delta(a_i,f^{(k)})}{2}}\},$$
$f(z)$ has a Borel direction with order $\geq\rho$.}

Recently, Zheng \cite{JH01} discussed the problem of $T$ directions
of a meromorphic function in one angular domain by proving.\\
\textbf{Theorem B.}\emph{ Let $f(z)$ be a transcendental meromorphic
function with finite lower order $\mu$ and non-zero order $\lambda$
and $f$ has a Nevanlinna deficient value $a\in\widehat{\mathbb{C}}$
with $\delta=\delta(a,f)>0$. For any positive and finite $\tau$ with
$\mu\leq\tau\leq\lambda$, consider the angular domain
$\Omega(\alpha,\beta)$ with
$$\beta-\alpha>\max\{\frac{\pi}{\tau},2\pi-\frac{4}{\tau}\arcsin\sqrt{\frac{\delta}{2}}\}.$$
Then $f(z)$ has a T direction in $\Omega=\Omega(\alpha,\beta)$.}

Following  Yang \cite{Yang01} and Zheng \cite{JH01}, we will
continue the discussion of singular directions of $f(z)$ in some
angular domains. The following three questions will be mainly
investigated in this paper.

\begin{que}Can we extend Theorem B to some angular domains
$$X=\bigcup\limits_{j=1}^q\{z:\alpha_j\leq\arg z\leq\beta_j\},$$
where the $q$ pair of real numbers $\{\alpha_j,\beta_j\}$ satisfy
\begin{equation}\label{05}
-\pi\leq\alpha_1<\beta_1\leq\alpha_2<\beta_2\leq\cdots\leq\alpha_q<\beta_q\leq\pi?
\end{equation}
\end{que}
\begin{que} Can $f(z)$ in Theorem B be replaced by any derivative
$f^{(p)}(z)(p\geq0)$?
\end{que}
\begin{que}What can we do if  $f(z)$ has many deficient values $a_1, a_2, a_3, \cdots,
a_l$ in Theorem B?
\end{que}

According to the Hayman inequality (see \cite{Hayman}) on the
estimation of $T(r,f)$ in terms of only two integrated counting
functions for the roots of $f(z)=a$ and $f^{(k)}(z)=b$ with
$b\not=0$, Guo, Zheng and Ng proposed in \cite{GuoZhengNg} a
singular direction named Hayman $T$ direction as follows.

\begin{defi}\label{def2}\ Let $f(z)$ be a transcendental meromorphic
function. A direction $L: \arg z=\theta$ is called a Hayman $T$
direction of $f(z)$ if for any small $\varepsilon>0$, any positive
integer $k$ and any complex numbers $a$ and $b\not=0$, we have
\begin{equation*}\label{4}
\limsup\limits_{r\longrightarrow\infty}\frac{N(r,Z_{\varepsilon}(\theta),f=a)
+N(r,Z_{\varepsilon}(\theta),f^{(k)}=b)}{ T(r,f)}>0.
\end{equation*}\end{defi}
Recently, Zheng and the first author \cite{Zheng02} confirmed the
existence of Hayman $T$ direction under the condition that
\begin{equation}\limsup\limits_{r\rightarrow+\infty}\frac{T(r,f)}{(\log r)^3}=+\infty
\end{equation}\\
In the same paper, the authors pointed out the Hayman $T$ direction
is different from  the $T$ direction and they gave an example to
show the growth condition (1.3) is sharp.  Can we discuss the
problem in some angular domains in the viewpoint of Question 1.1-1.3
? Though out this paper, we define
$$\omega=\max\{\frac{\pi}{\beta_1-\alpha_1},\cdots,\frac{\pi}{\beta_q-\alpha_q}\}.$$

Now, we state our theorems as follows.
\begin{thm}\label{thm1.1}
Let $f(z)$ be a transcendental meromorphic function with finite
lower order $\mu<\infty$, $0<\lambda\leq\infty$. There is an integer
$p\geq0$, such that $f^{(p)}$ has a Nevanlinna deficient value
$a\in\widehat{\mathbb{C}}$ with $\delta(a,f^{(p)})>0$. For $q$ pairs
of real numbers satisfies \eqref{05}. $f$ has at least one Hayman
$T$ direction in $X$ if
\begin{equation}\label{02}
\sum\limits_{j=1}^q(\alpha_{j+1}-\beta_j)<\frac{4}{\sigma}\arcsin\sqrt{\frac{\delta(a,f^{(p)})}{2}},
\end{equation}
where $\mu\leq\sigma\leq\lambda$, and $\omega<\sigma$.
\end{thm}

\begin{thm}\label{thm1.2}
Let $f(z)$ be a transcendental meromorphic function with finite
lower order $\mu<\infty$, $0<\lambda\leq\infty$. There is an integer
$p\geq0$, such that $f^{(p)}$ has $l\geq1$ distinct deficient values
$a_1, a_2, \cdots, a_l$ with the corresponding deficiency
$\delta(a_1,f^{(p)})$, $\delta(a_2,f^{(p)}),
\cdots,\delta(a_l,f^{(p)})$. For $q$ pair of real numbers
$\{\alpha_j,\beta_j\}$ satisfying \eqref{05} and
\begin{equation}\label{02}
\sum\limits_{j=1}^q(\alpha_{j+1}-\beta_j)<\sum\limits_{j=1}^l\frac{4}{\sigma}\arcsin\sqrt{\frac{\delta(a_j,f^{(p)})}{2}},
\end{equation}
where $\mu\leq\sigma\leq\lambda$. If $\omega<\sigma$, then $f$ has
at least one Hayman $T$ direction in $X$.
\end{thm}

We will only prove Theorem \ref{thm1.2}, and  Theorem \ref{thm1.1}
is a special case of Theorem \ref{thm1.2}.

\section{Primary knowledge and some lemmas}
In order to prove the theorems, we give some lemmas. The following
result is from \cite{Zheng}.
\begin{lem}
Let $f(z)$ be a transcendental meromorphic function with lower order
$\mu<\infty$ and order $0<\lambda\leq\infty$, then for any positive
number $\mu\leq\sigma\leq\lambda$ and a set $E$ with finite measure,
there exist a sequence $\{r_n\}$, such that

(1) $r_n\notin E$,
$\lim\limits_{n\rightarrow\infty}\frac{r_n}{n}=\infty$;

(2) $\liminf\limits_{n\rightarrow\infty}\frac{\log T(r_n,f)}{\log
r_n}\geq\sigma$;

(3) $T(t,f)<(1+o(1))(\frac{2t}{r_n})^\sigma
T(r_n/2,f),t\in[r_n/n,nr_n]$;

(4)$T(t,f)/t^{\sigma-\varepsilon_n}\leq2^{\sigma+1}T(r_n,f)/r_n^{\sigma-\varepsilon_n},1\leq
t\leq nr_n, \varepsilon_n=[\log n]^{-2}.$
\end{lem}

We recall that $\{r_n\}$ is called the P\'{o}lya peaks of order
$\sigma$ outside $E$. Given a positive function $\Lambda(r)$
satisfying $\lim_{r\rightarrow\infty}\Lambda(r)=0$. For $r>0$ and
$a\in\mathbb{C}$, define
$$D_\Lambda(r,a)=\{\theta\in[-\pi,\pi):\log^+\frac{1}{|f(re^{i\theta})-a|}>\Lambda(r)T(r,f)\},$$
and
$$D_\Lambda(r,\infty)=\{\theta\in[-\pi,\pi):\log^+|f(re^{i\theta})|>\Lambda(r)T(r,f)\}.$$
The following result is called the generalized spread relation, and
Wang in \cite{Wang} proved this.

\begin{lem}\label{lem03}
Let $f(z)$ be transcendental and meromorphic in $\mathbb{C}$ with
the finite lower order $\mu<\infty$ and the positive order
$0<\lambda\leq\infty$ and has $l\geq1$ distinct deficient values
$a_1, a_2, \cdots, a_l$. Then for any sequence of P\'{o}lya peaks
$\{r_n\}$ of order $\sigma>0,\mu\leq\sigma\leq\lambda$ and any
positive function $\Lambda(r)\rightarrow0$ as $r\rightarrow+\infty$,
we have
$$\liminf\limits_{n\rightarrow\infty}\sum\limits_{j=1}^l \meas D_\Lambda(r_n,a_j)\geq\min\{2\pi, \frac{4}{\sigma}\sum\limits_{j=1}^l \arcsin\sqrt{\frac{\delta(a_j,f^{(p)})}{2}}\}.$$
\end{lem}
From \cite{Yang01}, we know that for $a\neq b$ are two deficient
values of $f$, then we have $D_\Lambda(r,a)\bigcap
D_\Lambda(r,b)=\emptyset$.

Nevanlinna theory on the angular domain plays an important role in
this paper. Let us recall the following terms:
\begin{equation*}
\begin{split}
A_{\alpha,\beta}(r,f)&=\frac{\omega}{\pi}\int_1^r(\frac{1}
{t^\omega}-\frac{t^\omega}{r^{2\omega}})\{\log^+|f(te^{i\alpha})|+\log^+|f(te^{i\beta})|\}\frac{dt}{t},\\
B_{\alpha,\beta}(r,f)&=\frac{2\omega}{\pi
r^\omega}\int_{\alpha}^\beta\log^+|f(re^{i\theta})|\sin\omega(\theta-\alpha)d\theta,\\
C_{\alpha,\beta}(r,f)&=2\sum\limits_{1<|b_n|<r}(\frac{1}{|b_n|^\omega}-
\frac{|b_n|^\omega}{r^{2\omega}})\sin\omega(\theta_n-\alpha),
\end{split}
\end{equation*}
where $\omega=\frac{\pi}{\beta-\alpha}$, and
$b_n=|b_n|e^{i\theta_n}$ is a pole of $f(z)$ in the angular domain
$\Omega(\alpha,\beta)$, appeared according to the multiplicities.
The Nevanlinna's angular characteristic is defined as follows:
\begin{equation*}
S_{\alpha,\beta}(r,f)=A_{\alpha,\beta}(r,f)+B_{\alpha,\beta}(r,f)+C_{\alpha,\beta}(r,f).
\end{equation*}
From the definition of $B_{\alpha,\beta}(r,f)$, we have the
following inequality, which will be used in the next.
\begin{equation}\label{06}
B_{\alpha,\beta}(r,f)\geq\frac{2\omega\sin(\omega\varepsilon)}{\pi
r^\omega}\int_{\alpha+\varepsilon}^{\beta-\varepsilon}\log^+|f(re^{i\theta})|d\theta
\end{equation}
The following is  the Nevanlinna first and second fundamental
theorem on the angular domains.

\begin{lem}\label{notation01}
Let $f$ be a nonconstant meromorphic function on the angular domain
$\Omega(\alpha,\beta)$. Then for any complex number $a$,
\begin{equation*}
S_{\alpha,\beta}(r,f)=S_{\alpha,\beta}(r,\frac{1}{f-a})+O(1),
r\rightarrow\infty,
\end{equation*}
and for any $q(\geq3)$ distinct points $a_j\in\widehat{\mathbb{C}}\
(j=1,2,\ldots,q)$,
\begin{equation*}\label{nevan01}
\begin{split}
(q-2)S_{\alpha,\beta}(r,f)&\leq\sum\limits_{j=1}^q\overline{C}_{\alpha,\beta}(r,\frac{1}{f-a_j})+Q_{\alpha,\beta}(r,f),
\end{split}
\end{equation*}
where
$$Q_{\alpha,\beta}(r,f)=(A+B)_{\alpha,\beta}(r,\frac{f'}{f})
+\sum\limits_{j=1}^q(A+B)_{\alpha,\beta}(r,\frac{f'}{f-a_j})+O(1).$$
\end{lem}
The key point is the estimation of error term
$Q_{\alpha,\beta}(r,f)$, which can be obtained for our purpose of
this paper as follows. And the following is true(see
\cite{Goldberg}). Write
$$Q(r,f)=A_{\alpha,\beta}(r,\frac{f^{(p)}}{f})+B_{\alpha,\beta}(r,\frac{f^{(p)}}{f}).$$
Then

(1)$Q(r,f)=O(\log r)$ as $r\rightarrow\infty$, when
$\lambda(f)<\infty$.

(2)$Q(r,f)=O(\log r+\log T(r,f))$ as $r\rightarrow\infty$ and
$r\notin E$ when $\lambda(f)=\infty$, where $E$ is a set with finite
linear measure.

The following result is useful for our study, the proof of which is
similar to the case of the characteristic function $T(r,f)$ and
$T(r,f^{(k)})$ on the whole complex plane. For the completeness, we
give out the proof.

\begin{lem}\label{lem01}
Let $f(z)$ be a meromorphic function on the whole complex plane.
Then for any angular domain $\Omega(\alpha,\beta)$, we have
$$S_{\alpha,\beta}(r,f^{(p)})\leq(p+1)S_{\alpha,\beta}(r,f)+O(\log r+\log T(r,f)),$$
possibly outside a set of $r$ with finite measure.
\end{lem}
\begin{proof}
In view of the definition of $S_{\alpha,\beta}(r,f)$ and Lemma
\ref{notation01}, we get the following
\begin{equation*}
\begin{split}
S_{\alpha,\beta}(r,f^{(p)})&\leq
C_{\alpha,\beta}(r,f^{(p)})+(A+B)_{\alpha,\beta}(r,f)+(A+B)_{\alpha,\beta}(r,\frac{f^{(p)}}{f})\\
&=p\overline{C}_{\alpha,\beta}(r,f)+S_{\alpha,\beta}(r,f)+(A+B)_{\alpha,\beta}(r,\frac{f^{(p)}}{f})\\
&\leq(p+1)S_{\alpha,\beta}(r,f)+Q(r,f).
\end{split}
\end{equation*}
\end{proof}

Recall the definition of Ahlfors-Shimizu characteristic in an angle
(see \cite{Tsuji}). Let $f(z)$ be a meromorphic function on an angle
$\Omega=\{z:\alpha\leq\arg z\leq \beta\}$. Set
$\Omega(r)=\Omega\cap\{z:1<|z|<r\}$. Define
$$\mathcal{S}(r,\Omega,f)=\frac{1}{\pi}\int\int_{\Omega(r)}{\left(|f'(z)|\over
1+|f(z)|^2\right)^2}d\sigma$$ and
$$\mathcal{T}(r,\Omega,f)=\int_1^r{\mathcal{S}(t,\Omega,f)\over
t}dt.$$

The following lemma is a theorem in \cite{Zheng02}, which is to
controll the term $\mathcal {T}(r,\Omega_\varepsilon)$ using the
counting functions $N(r,\Omega,f=a)$ and $N(r,\Omega,f^{(k)}=b)$.
\begin{lem}\label{04}
Let $f(z)$ be meromorphic in an angle $\Omega=\{z:\alpha\leq\arg z
\leq\beta\}$. Then for any small $\varepsilon>0$, any positive
integer $k$ and any two complex numbers $a$ and $b\not=0$, we have
\begin{equation}\label{9}
\mathcal{T}(r,\Omega_\varepsilon,f) \leq K\{N(2r,\Omega,f=a)+
N(2r,\Omega,f^{(k)}=b)\}+O(\log^3 r)
\end{equation} for a positive constant $K$ depending only on $k$, where $\Omega_\varepsilon=\{z:\alpha+\varepsilon<\arg
z<\beta-\varepsilon\}$.
\end{lem} In order to prove our theorem, we have to use the following
lemma, which is a consequent result of Theorem 3.1.6 in \cite{JH01}.
\begin{lem}
Let $f(z)$ be a transcendental meromorphic function in the whole
plane, and satisfies the conditions of Theorem \ref{thm1.2} or
Theorem \ref{thm1.1}. Take a sequence of P\'{o}lya peak $\{r_n\}$ of
$f(z)$ of order $\sigma>\omega=\frac{\pi}{\beta-\alpha}$. If $f(z)$
has no Hayman T direction in the angular domain
$\Omega(\alpha,\beta)$, then the following real function satisfy
$\lim\limits_{r\rightarrow\infty}\Lambda(r)=0$, which $\Lambda(r)$
is defined as follows
$$\Lambda(r)^2=\max\{\frac{\mathcal
{T}(r_n,\Omega_\varepsilon,f)}{T(r_n,f)},
\frac{r_n^{\omega}}{T(r_n,f)}\int_1^{r_n}\frac{\mathcal
{T}(t,\Omega_\varepsilon,f)}{t^{\omega+1}}dt,\frac{r_n^\omega[\log
r_n+\log T(r_n,f)]}{T(r_n,f)}\},$$ for $r_n\leq r<r_{n+1}.$
\end{lem}
\begin{proof} We should treat two cases.\\
Case (I). If there is no Hayman $T$ direction on $\Omega$, then from
Lemma \ref{04}, we have
$$\mathcal {T}(r,\Omega_\varepsilon,f)=o(T(2r,f))+O(\log^3r), \ as\  r\rightarrow\infty.$$
Combining Lemma 2.1 and $\sigma>\omega$, we have
\begin{equation*}
\begin{split}
\int_1^{r_n}\frac{\mathcal
{T}(t,\Omega_\varepsilon,f)}{t^{\omega+1}}dt&=o(\int_1^{r_n}\frac{T(2t,f)}{t^{\omega+1}}dt)+\int_1^{r_n}\frac{O(\log^3
t)}{t^{\omega+1}}dt\\
&\leq
o(\int_1^{r_n}\frac{T(r_n,f)}{t^{\omega+1}}(\frac{2t}{r_n})^\sigma
dt)+O(\log^3r_n)\\
&=o(\frac{T(r_n,f)}{r_n^\omega})+O(\log^3r_n)\\
\end{split}
\end{equation*}
Then
$$\frac{r_n^\omega}{T(r_n,f)}\int_1^{r_n}\frac{\mathcal {T}(t,\Omega_\varepsilon)}{t^{\omega+1}}dt\rightarrow0, \ as \ n\rightarrow\infty. $$
Case (II).  If $$\limsup\limits_{n\rightarrow\infty}\frac{\mathcal
{T}(r_n,\Omega_\varepsilon,f)}{T(r_n,f)}>0,$$ then by \eqref{9}, we
have
$$\limsup\limits_{n\rightarrow\infty}\frac{N(2r_n,\Omega,f=a)+N(2r_n,\Omega,f^{(k)}=b)}{T(r_n,f)}>0.$$
Since $\{r_n\}$ is a sequence of P\'{o}lya peaks of order $\sigma$,
then we have $$T(2r_n,f)\leq2^\sigma T(r_n,f).$$ Then $\Omega$ must
contain a Hayman $T$ direction of $f(z)$. This is contradict to the
hypothesis.

From Case (I) and Case (II) and notice that $r_n^\omega[\log
r_n+\log T(r_n,f)]/T(r_n,f)\rightarrow0,(n\rightarrow\infty)$, we
have proved that $\limsup_{r\rightarrow\infty}\Lambda(r)=0$.

\end{proof}
The following result was firstly established by Zheng
\cite{JH01}(Theorem 2.4.7), it is crucial for our study.
\begin{lem}
Let $f(z)$ be a function meromorphic on
$\Omega=\Omega(\alpha,\beta)$. Then
$$S_{\alpha,\beta}(r,f)\leq2\omega^2\frac{\mathcal {T}(r,\Omega,f)}{r^\omega}+\omega^3\int_1^r\frac{\mathcal {T}(t,\Omega,f)}{t^{\omega+1}}dt+O(1),\ \ \omega=\frac{\pi}{\beta-\alpha}.$$
\end{lem}

We also have to use the following lemma, which is due to Hayman and
Miles \cite{Miles}.
\begin{lem}\label{lem02}
Let $f(z)$ be meromorphic in the complex plane. Then for a given
$K>1$, there exists a set $M(K)$ with $\overline{\log
dens}M(K)\leq\delta(K)$,
$\delta(K)=\min\{(2e^{K-1}-1)^{-1},(1+e(K-1)exp(e(1-K)))\}$, such
that
$$\limsup\limits_{r\rightarrow+\infty,r\notin M(K)}\frac{T(r,f)}{T(r,f^{(p)})}\leq3eK.$$
\end{lem}

\section{Proof of theorem \ref{thm1.2}}

\begin{proof}
Case(I). $\lambda(f)>\mu$. Then we choose $\sigma$ such that
$\lambda(f^{(p)})=\lambda(f)>\sigma\geq\mu=\mu(f^{(p)}),
\sigma>\omega$. From the inequality \eqref{02}, we can take a real
number $\varepsilon>0$ such that
\begin{equation}\label{01}
\sum\limits_{j=1}^q(\alpha_{j+1}-\beta_j+4\varepsilon)+\varepsilon<\sum\limits_{j=1}^l\frac{4}{\sigma+2\varepsilon}\arcsin\sqrt{\frac{\delta(a_j,f^{(p)})}{2}},
\end{equation}
and $$\lambda(f^{(p)})>\sigma+2\varepsilon>\mu.$$ Then there exists
a sequence of P\'{o}lya peaks $\{r_n\}$ of order
$\sigma+2\varepsilon$ of $f^{(p)}$ such that $\{r_n\}$ are not in
the set of Lemma \ref{lem01} and Lemma \ref{lem02}.

We define $q$ real functions $\Lambda_j(r)(j=1,2,\cdots,q)$ as
follows.
\begin{equation*}
\begin{split}
\Lambda_j(r)^2=\max\{&\frac{\mathcal
{T}(r_n,\Omega(\alpha_j+\varepsilon,\beta_j-\varepsilon),f)}{T(r_n,f)},\\
&\frac{r_n^{\omega_j}}{T(r_n,f)}
 \int_1^{r_n}\frac{\mathcal
{T}(t,\Omega(\alpha_j+\varepsilon,\beta_j-\varepsilon),f)}{t^{\omega_j+1}}dt,\frac{r_n^{\omega_j}[\log
r_n+\log T(r_n,f)]}{T(r_n,f)}\},
\end{split}
\end{equation*}
for $r_n\leq r<r_{n+1},
\omega_{j}=\frac{\pi}{\beta_{j}-\alpha_{j}}$. By using Lemma 2.5, we
have $\Lambda_j(r)\rightarrow0$, as $r\rightarrow\infty$, if $f(z)$
has no Hayman $T$ directions on $X$. Set $\Lambda(r)=\max_{1\leq
j\leq q}\{\Lambda_j(r)\}$, we have
$\lim_{r\rightarrow\infty}\Lambda(r)=0$. Therefore for large enough
$n$, by Lemma \ref{lem03} we have
\begin{equation}\label{03}
\sum\limits_{j=1}^l \meas D_\Lambda(r_n,a_j)>\min\{2\pi,
\frac{4}{\sigma+2\varepsilon}\sum\limits_{j=1}^l
\arcsin\sqrt{\frac{\delta(a_j,f^{(p)})}{2}}\}-\varepsilon.
\end{equation}
We note that $\sigma+2\varepsilon>1/2$, we suppose for any $n$
\eqref{03} holds. Set
$$K_n=\meas
((\bigcup\limits_{j=1}^lD_\Lambda(r_n,a_j))\bigcap(\bigcup\limits_{j=1}^q(\alpha_j+2\varepsilon,\beta_j-2\varepsilon))).$$
Combining \eqref{01} with \eqref{03}, we obtain
\begin{equation*}
\begin{split}
K_n&\geq \sum\limits_{j=1}^l\meas (D_\Lambda(r_n,
a_j))-\meas([-\pi,\pi)\backslash\bigcup\limits_{j=1}^q(\alpha_j+2\varepsilon,\beta_j-2\varepsilon))\\
&= \sum\limits_{j=1}^l\meas (D_\Lambda(r_n,
a_j))-\meas(\bigcup\limits_{j=1}^q(\beta_j-2\varepsilon,\alpha_{j+1}+2\varepsilon))\\
&=\sum\limits_{j=1}^l\meas (D_\Lambda(r_n,
a_j))-\sum\limits_{j=1}^q(\alpha_{j+1}-\beta_j+4\varepsilon)>\varepsilon>0.
\end{split}
\end{equation*} It is easy to see that, there exists a $j_0$ such that for
infinitely many $n$, we have
\begin{equation*}
\meas
(\bigcup\limits_{j=1}^lD_\Lambda(r_n,a_j)\bigcap(\alpha_{j_0}+2\varepsilon,\beta_{j_0}-2\varepsilon))>\frac{K_n}{q}>\frac{\varepsilon}{q}.
\end{equation*}
We can assume that the above holds for all the $n$.

Set
$E_{nj}=D(r_n,a_j)\bigcap(\alpha_{j_0}+2\varepsilon,\beta_{j_0}-2\varepsilon)$.
Thus we have
\begin{equation}\label{07}
\begin{split}
\sum\limits_{j=1}^l\int_{\alpha_{j_0}+2\varepsilon}^{\beta_{j_0}-2\varepsilon}\log^+\frac{1}{|f^{(p)}(r_ne^{i\theta})-a_j|}d\theta&\geq
\sum\limits_{j=1}^l\int_{E_{nj}}\log^+\frac{1}{|f^{(p)}(r_ne^{i\theta})-a_j|}d\theta\\
&\geq\sum\limits_{j=1}^l\meas(E_{nj})\Lambda(r_n)T(r_n,f^{(p)})\\&>\frac{\varepsilon}{q}\Lambda(r_n)T(r_n,f^{(p)})\\
&>\frac{\varepsilon}{3eqK}\Lambda(r_n)T(r_n,f).
\end{split}
\end{equation}
The last inequality uses Lemma 2.8.

On the other hand, we have
\begin{equation}\label{08}
\begin{split}
&\sum\limits_{j=1}^l\int_{\alpha_{j_0}+2\varepsilon}^{\beta_{j_0}-2\varepsilon}\log^+\frac{1}{|f^{(p)}(r_ne^{i\theta})-a_j|}d\theta\leq\sum\limits_{j=1}^l
\frac{\pi}{2\omega_{j_0}\sin(\varepsilon\omega_{j_0})}r_n^{\omega_{j_0}}B_{\alpha_{j_0}+\varepsilon,\beta_{j_0}-\varepsilon}(r_n,\frac{1}{f^{(p)}-a_j})\\
&<\sum\limits_{j=1}^l\frac{\pi}{2\omega_{j_0}\sin(\varepsilon\omega_{j_0})}r_n^{\omega_{j_0}}S_{\alpha_{j_0}+\varepsilon,\beta_{j_0}-\varepsilon}(r_n,\frac{1}{f^{(p)}-a_j})\\
&=\frac{l\pi}{2\omega_{j_0}\sin(\varepsilon\omega_{j_0})}r_n^{\omega_{j_0}}S_{\alpha_{j_0}+\varepsilon,\beta_{j_0}-\varepsilon}(r_n,f^{(p)})+O(r_n^{\omega_{j_0}})\\
&\leq\frac{l\pi}{2\omega_{j_0}\sin(\varepsilon\omega_{j_0})}r_n^{\omega_{j_0}}[(p+1)S_{\alpha_{j_0}+\varepsilon,\beta_{j_0}-\varepsilon}(r_n,f)+\log r_n+\log T(r_n,f)]+O(r_n^{\omega_{j_0}})\\
&\leq\frac{l\pi}{2\omega_{j_0}\sin(\varepsilon\omega_{j_0})}(p+1)[2\omega_{j_0}^2\mathcal
{T}(r_n,\Omega(\alpha_{j_0}+\varepsilon,\beta_{j_0}-\varepsilon),f)\\&+\omega_{j_0}^3r_n^{\omega_{j_0}}\int_1^{r_n}\frac{\mathcal
{T}(t,\Omega(\alpha_{j_0}+\varepsilon,\beta_{j_0}-\varepsilon),f)}{t^{\omega_{j_0}+1}}dt]\\
&+\frac{l\pi}{2\omega_{j_0}\sin(\varepsilon\omega_{j_0})}r_n^{\omega_{j_0}}[\log r_n+\log T(r_n,f)]+O(r_n^{\omega_{j_0}})\\
&\leq\frac{l\pi}{2\omega_{j_0}\sin(\varepsilon\omega_{j_0})}(p+1)[2\omega_{j_0}^2\Lambda(r_n)^2T(r_n,f)
+\omega_{j_0}^3\Lambda(r_n)^2T(r_n,f)]\\
&+\frac{l\pi}{2\omega_{j_0}\sin(\varepsilon\omega_{j_0})}r_n^{\omega_{j_0}}[\log
r_n+\log T(r_n,f)]+O(r_n^{\omega_{j_0}}),\ \ \ \
\omega_{j_0}=\frac{\pi}{\beta_{j_0}-\alpha_{j_0}-2\varepsilon}. \
\end{split}
\end{equation}
\eqref{07} and \eqref{08} imply that
$$\Lambda(r_n)\leq O(\Lambda(r_n)^2).$$
A contradiction is derived because $\Lambda(r_n)\rightarrow0$ as
$n\rightarrow\infty$.

Case (II). $\lambda(f)=\mu$. By the same argument as in Case1 with
all the $\sigma+2\varepsilon$ replaced by $\sigma=\mu$, we can
derive the same contradiction.
\end{proof}

 \end{document}